\documentclass[9pt,twocolumn,twoside]{pnas-new}
\templatetype{pnasresearcharticle} % Choose template 
\usepackage{svg}
\usepackage{bm}
\usepackage{mathtools}
\usepackage{amsmath}

\usepackage{wasysym}
\usepackage{mathdots}
\DeclareRobustCommand
  \Compactcdots{\mathinner{\cdotp\mkern-3mu\cdotp\mkern-3mu\cdotp}}
\algrenewcommand\algorithmicrequire{\textbf{Input:}}
\algrenewcommand\algorithmicensure{\textbf{Output:}}
\DeclareRobustCommand
  \Compactcdots{\mathinner{\cdotp\mkern-3mu\cdotp\mkern-3mu\cdotp}}
\algrenewcommand\algorithmicrequire{\textbf{Input:}}
\algrenewcommand\algorithmicensure{\textbf{Output:}}
\usepackage{amsthm}
\usepackage{amsmath}
\newtheorem{theorem}{Theorem}[subsection]

\theoremstyle{remark}

\newtheorem*{nomenclature}{Nomenclature}

\usepackage{mathtools}
\usepackage{amssymb}

\usepackage {mathtools}
\usepackage{amsmath}
\usepackage{algpseudocode}
\usepackage{algorithm}\usepackage{CJKutf8}
\algnewcommand\algorithmicforeach{\textbf{for each}}
\algdef{S}[FOR]{ForEach}[1]{\algorithmicforeach\ #1\ \algorithmicdo}
\DeclareRobustCommand
  \Compactcdots{\mathinner{\cdotp\mkern-3mu\cdotp\mkern-3mu\cdotp}}

\title{Robust estimations from distribution structures: II. Central Moments}
\author[a,b,c,1]{Li Tuobang}
\leadauthor{Li}
\significancestatement{

In nonparametric statistics, the focus is on the relative differences of robust estimators, which is considered more crucial than their precise values. This principle implies that if the underlying distribution's parameters shift, then all corresponding nonparametric estimates are expected to uniformly and asymptotically adjust in a consistent direction, provided they target the same characteristic of the distribution. This article discusses the validity of this fundamental principle of nonparametrics in various scenarios. It is found that for the $\mathbf{k}$th central moment, kernel distributions generally follow this principle.

}
\authorcontributions{L.T. designed research, performed research, analyzed data, and wrote the paper.}
\authordeclaration{The author declares no competing interest.}
\correspondingauthor{\textsuperscript{1}To whom correspondence should be addressed. E-mail: tuobang@biomathematics.org}
\keywords{moments $|$ invariant $|$ unimodal $|$ $U$-statistics $|$ generalized $L$-statistics} 
\begin{abstract}
In descriptive statistics, $U$-statistics arise naturally in producing minimum-variance unbiased estimators. In 1984, Serfling considered the distribution formed by evaluating the kernel of the $U$-statistics and proposed generalized $L$-statistics which includes Hodges-Lehamnn estimator and Bickel-Lehmann spread as special cases. However, the structures of the kernel distributions remain unclear. In 1954, Hodges and Lehmann demonstrated that if $X$ and $Y$ are independently sampled from the same unimodal distribution, $X-Y$ will exhibit symmetrical unimodality with its peak centered at zero. Building upon this foundational work, the current study delves into the structure of the kernel distribution. It is shown that the $\mathbf{k}$th central moment kernel distributions ($\mathbf{k}>2$) derived from a unimodal distribution exhibit location invariance and is also nearly unimodal with the mode and median close to zero. This article provides an approach to study the general structure of kernel distributions and illuminates the understanding of the common nature of the measures of probability distributions.
 
\end{abstract}

\dates{This manuscript was compiled on \today}
\begin{document}

\maketitle
\thispagestyle{firststyle}
\ifthenelse{\boolean{shortarticle}}{\ifthenelse{\boolean{singlecolumn}}{\abscontentformatted}{\abscontent}}{}

\dropcap{T}he most popular robust scale estimator currently, the median absolute deviation, was popularized by Hampel (1974) \cite{hampel1974influence}, who credits the idea to Gauss in 1816 \cite{gauss1816bestimmung}. It can be seen as evaluating the median of a pseudo-sample formed by the absolute deviations of all values related to the sample median. The pseudo-sample size is $n$. Indeed, most scale estimators can be transformed in such ways. For example, range or interquartile range can be seen as evaluating the mean of a pseudo-sample with two values, and they belong to the class of scale estimators called quantile differences. In 1976, in their landmark series \emph{Descriptive Statistics for Nonparametric Models}, Bickel and Lehmann \cite{bickel2012descriptive3} generalized a class of estimators as measures of the dispersion of a symmetric distribution around its center of symmetry. Median absolute deviation, sample variance, and average absolute deviation are all belonging to this class. In 1979, the same series, they \cite{bickel2012descriptive} proposed a class of estimators referred to as measures of spread, which consider the pairwise differences of a random variable, irrespective of its symmetry, throughout its distribution, rather than focusing on dispersion relative to a fixed point. In the final section \cite{bickel2012descriptive}, they explored a version of the trimmed standard deviation based on $n^2$ pairwise differences, which is modified here for comparison, 

\begin{equation}
\begin{split}
\left[n\left(\frac{1}{2}-\epsilon\right)\right]^{-\frac{1}{2}}\left[\sum_{i=\frac{n}{2}}^{n\left(1-\epsilon\right)}\left[X_i-X_{n-i+1}\right]^2\right]^\frac{1}{2} 
\text{,}\label{eq:1}
\end{split}
\end{equation}
and

\begin{equation}\left[\binom{n}{2}\left(1-\epsilon_\mathbf{0}-\gamma\epsilon_\mathbf{0}\right)\right]^{-\frac{1}{2}}\left[\sum_{i=\binom{n}{2}\gamma\epsilon_\mathbf{0}}^{\binom{n}{2}\left(1-\epsilon_\mathbf{0}\right)}{\left(X_{i_1}-X_{i_2}\right)_i}^2\right]^\frac{1}{2}\text{,}\label{eq:2}\end{equation}
where $\left(X_{i_1}-X_{i_2}\right)_1\le\ldots\le\left(X_{i_1}-X_{i_2}\right)_{\binom{n}{2}}$ are the order statistics of the pseudosample, $X_{i_1}-X_{i_2}$, ${i_1}<{i_2}$, provided that $\binom{n}{2}\gamma\epsilon_\mathbf{0}\in\mathbb{N}$ and $\binom{n}{2}\left(1-\epsilon_\mathbf{0}\right)\in\mathbb{N}$. They showed that, when $\epsilon_\mathbf{0}=0$, the result obtained using [\ref{eq:2}] is equal to $\sqrt2$ times the sample standard deviation. The paper ended with, “We do not know a fortiori which of the measures is preferable and leave these interesting questions open.” 

Two examples of the impacts of that series are as follows. Oja (1981, 1983) \cite{oja1981location,oja1983descriptive} provided a more comprehensive and generalized examination of these concepts, and integrated the measures of location, dispersion, and spread as proposed by Bickel and Lehmann \cite{bickel2012descriptive2,bickel2012descriptive3,bickel2012descriptive}, along with van Zwet's convex transformation order of skewness and kurtosis (1964) \cite{von2012convex} for univariate and multivariate distributions, resulting a greater degree of generality and a broader perspective on these statistical constructs. Rousseeuw and Croux proposed a popular efficient scale estimator based on separate medians of pairwise differences taken over ${i_1}$ and ${i_2}$ \cite{rousseeuw1993alternatives} in 1993. However the importance of tackling the symmetry assumption has been greatly underestimated, as will be discussed later. 

Here, their open question is addressed in two different aspects \cite{bickel2012descriptive}. First, since the estimation of scale can be transformed into the location estimation of a pseudo-sample, according to the principle of the central limit theorem, the variances of such scale estimators should be linearly dependent on the standard deviation of the pseudo-sample and inversely dependent on the square root of the pseudo-sample size. Then, [\ref{eq:2}] based on $n^2$ pairwise differences is obviously better than [\ref{eq:1}] since the ratio of its pseudo-sample size over that of [\ref{eq:1}] is $n$. So if just considering the size, the variance of [\ref{eq:2}] is $\frac{1}{\sqrt{n}}$ of the variance of [\ref{eq:1}]. Another factor that needs to be considered is the standard deviation of the pseudo-sample. However, the standard deviation of the pseudo-sample is generally independent of different pseudo-sample sizes. So, no matter how different the standard deviation of the pseudo-samples of [\ref{eq:1}] and [\ref{eq:2}] is, as the sample size increases, the variance of [\ref{eq:1}] will always dominate that of [\ref{eq:2}]. Second, the nomenclature used in this series is introduced as follows:\begin{nomenclature}\label{Nomen} Given a robust estimator, $\hat{\theta}$, which has an adjustable breakdown point, $\epsilon$, that can approach zero asymptotically, the name of $\hat{\theta}$ comprises two parts: the first part denotes the type of estimator, and the second part represents the population parameter $\theta$, such that $\hat{\theta}\rightarrow\theta$ as $\epsilon\rightarrow0$. The abbreviation of the estimator combines the initial letters of the first part and the second part. If the estimator is symmetric, the upper asymptotic breakdown point, ${\epsilon}$, is indicated in the subscript of the abbreviation of the estimator, with the exception of the median. For an asymmetric estimator based on quantile average, the associated $\gamma$ follows ${\epsilon}$. \end{nomenclature} 

In REDS I \cite{REDS1}, it was shown that the bias of a robust estimator with an adjustable breakdown point is often monotonic with respect to the breakdown point in a semiparametric distribution. Naturally, the estimator's name should reflect the population parameter that it approaches as $\epsilon\rightarrow0$. If multiplying all pseudo-samples by a factor of $\frac{1}{\sqrt{2}}$, then [\ref{eq:2}] is the trimmed standard deviation adhering to this nomenclature, since $\psi_{2}\left(x_1,x_2\right)=\frac{1}{2}\left(x_1-x_2\right)^2$ is the kernel function of the unbiased estimation of the second central moment by using $U$-statistic \cite{heffernan1997unbiased}. This definition should be preferable, not only because it is the square root of a trimmed $U$-statistic, which is closely related to the minimum-variance unbiased estimator (MVUE), but also because the second $\gamma$-orderliness of the second central moment kernel distribution is ensured by the next exciting theorem. 

\begin{theorem}\label{scmkdti}The second central moment kernel distribution generated from any unimodal distribution is second $\gamma$-ordered, provided that $\gamma\geq0$. \end{theorem}\begin{proof}In 1954, Hodges and Lehmann established that if $X$ and $Y$ are independently drawn from the same unimodal distribution, $X-Y$ will be a symmetric unimodal distribution peaking at zero \cite{hodges1954matching}. Given the constraint in the pairwise differences that $X_{i_1}<X_{i_2}$, ${i_1}<{i_2}$, it directly follows from Theorem 1 in \cite{hodges1954matching} that the pairwise difference distribution ($\Xi_{\Delta}$) generated from any unimodal distribution is always monotonic increasing with a mode at zero. Since $X-X^\prime$ is a negative variable that is monotonically increasing, applying the squaring transformation, the relationship between the original variable \(X-X^\prime\) and its squared counterpart \( (X-X^\prime)^2 \) can be represented as follows: $X-X^\prime < Y-Y^\prime \implies (X-X^\prime)^2 > (Y-Y^\prime)^2$. In other words, as the negative values of \(X-X^\prime\) become larger in magnitude (more negative), their squared values \( (X-X^\prime)^2 \) become larger as well, but in a monotonically decreasing manner with a mode at zero. Further multiplication by $\frac{1}{2}$ also does not change the monotonicity and mode, since the mode is zero. Therefore, the transformed pdf becomes monotonically decreasing with a mode at zero. In REDS I \cite{REDS1}, it was proven that a right-skewed distribution with a monotonic decreasing pdf is always second $\gamma$-ordered, which gives the desired result.\end{proof}

In REDS I \cite{REDS1}, it was shown that any symmetric distribution is $\nu$th $U$-ordered, suggesting that $\nu$th $U$-orderliness does not require unimodality, e.g., a symmetric bimodal distribution is also $\nu$th $U$-ordered. In the SI Text of REDS I \cite{REDS1}, an analysis of the Weibull distribution showed that unimodality does not assure orderliness. Theorem \ref{scmkdti} uncovers a profound relationship between unimodality, monotonicity, and second $\gamma$-orderliness, which is sufficient for $\gamma$-trimming inequality and $\gamma$-orderliness. 

On the other hand, while robust estimation of scale has been intensively studied with established methods \cite{bickel2012descriptive3,bickel2012descriptive}, the development of robust measures of asymmetry and kurtosis lags behind, despite the availability of several approaches \cite{bowley1926elements,van1964convex,groeneveld1984measuring,saw1958moments,elamir2003trimmed}. The purpose of this paper is to demonstrate that, in light of previous works, the estimation of all central moments can be transformed into a location estimation problem by using $U$-statistics and the central moment kernel distributions possess desirable properties.%, and define a convenient approach to quantitatively estimate the estimators' efficiencies. 
\section*{Robust Estimations of the Central Moments}\label{BB}
In 1928, Fisher constructed $\mathbf{k}$-statistics as unbiased estimators of cumulants \cite{fisher1930moments}. Halmos (1946) proved that a functional $\theta$ admits an unbiased estimator if and only if it is a regular statistical functional of degree $\mathbf{k}$ and showed a relation of symmetry, unbiasness and minimum variance \cite{halmos1946theory}. Hoeffding, in 1948, generalized $U$-statistics \cite{hoeffding1948class} which enable the derivation of a minimum-variance unbiased estimator from each unbiased estimator of an estimable parameter. In 1984, Serfling pointed out the speciality of Hodges-Lehmann estimator, which is neither a simple $L$-statistic nor a $U$-statistic, and considered the generalized $L$-statistics and trimmed $U$-statistics \cite{serfling1984generalized}. Given a kernel function $h_{\mathbf{k}}$ which is a symmetric function of $\mathbf{k}$ variables, the $LU$-statistic is defined as:\begin{align*}LU_{h_{\mathbf{k}},\mathbf{k},k,\epsilon,\gamma,n}\coloneqq LL_{k,\epsilon_\mathbf{0},\gamma,n}\left(\text{sort}\left(\left(h_{\mathbf{k}}\left(X_{N_1},\ldots,X_{N_\mathbf{k}}\right)\right)_{N=1}^{\binom{n}{\mathbf{k}}}\right)\right)\text{,}\end{align*}where $\epsilon=1-\left(1-\epsilon_\mathbf{0}\right)^\frac{1}{\mathbf{k}}$ (proven in REDS III \cite{REDS3}), $X_{N_1},\ldots,X_{N_\mathbf{k}}$ are the $n$ choose $\mathbf{k}$ elements from the sample, $LL_{k,\epsilon_\mathbf{0},\gamma,n}(Y)$ denotes the $LL$-statistic with the sorted sequence $\text{sort}\left(\left(h_{\mathbf{k}}\left(X_{N_1},\ldots,X_{N_\mathbf{k}}\right)\right)_{N=1}^{\binom{n}{\mathbf{k}}}\right)$ serving as an input. In the context of Serfling's work, the term ‘trimmed $U$-statistic’ is used when $LL_{k,\epsilon_\mathbf{0},\gamma,n}$ is $\text{TM}_{\epsilon_\mathbf{0},\gamma,n}$ \cite{serfling1984generalized}.

In 1997, Heffernan \cite{heffernan1997unbiased} obtained an unbiased estimator of the $\mathbf{k}$th central moment by using $U$-statistics and demonstrated that it is the minimum variance unbiased estimator for distributions with the finite first $\mathbf{k}$ moments. The weighted H-L $\mathbf{k}$th central moment ($2\le \mathbf{k}\le n$) is thus defined as, \begin{align*}\text{WHL}\mathbf{k}m_{k,\epsilon,\gamma,n}\coloneqq LU_{h_{\mathbf{k}}=\psi_\mathbf{k},\mathbf{k},k,\epsilon,\gamma,n}\text{,}\end{align*}where $\text{WHLM}_{k,\epsilon_\mathbf{0},\gamma,n}$ is used as the $LL_{k,\epsilon_\mathbf{0},\gamma,n}$ in $LU$, $\psi_\mathbf{k}\left(x_1,\ldots,x_\mathbf{k}\right)=\sum_{j=0}^{\mathbf{k}-2}{\left(-1\right)^j\left(\frac{1}{\mathbf{k}-j}\right)\sum\left(x_{i_1}^{\mathbf{k}-j}x_{i_2}\ldots x_{i_{j+1}}\right)}+\left(-1\right)^{\mathbf{k}-1}\left(\mathbf{k}-1\right)x_1\ldots x_\mathbf{k}$, the second summation is over $i_1,\ldots,i_{j+1}=1$ to $\mathbf{k}$ with $i_1\neq i_2\neq\ldots\neq i_{j+1}$ and $i_2<i_3<\ldots<i_{j+1}$ \cite{heffernan1997unbiased}. Despite the complexity, the following theorem offers an approach to infer the general structure of such kernel distributions.

\begin{theorem}\label{tfkumo}
Define a set $T$ comprising all pairs $(\psi_\mathbf{k}(\mathbf{v}), f_{X,\ldots,X}(\mathbf{v}))$ such that $\psi_\mathbf{k}(\mathbf{v}) = \psi_\mathbf{k}\left(Q(p_{1}),\ldots,Q(p_{\mathbf{k}})\right)$ with $Q(p_{1})<\ldots<Q(p_{\mathbf{k}})$ and $f_{X,\ldots,X}(\mathbf{v})=\mathbf{k}!f(Q(p_{1}))\ldots f(Q(p_{\mathbf{k}}))$ is the probability density of the $\mathbf{k}$-tuple, $\mathbf{v}=(Q(p_{1}),\ldots,Q(p_{\mathbf{k}}))$ (a formula drawn after a modification of the Jacobian density theorem). $T_\Delta$ is a subset of $T$, consisting all those pairs for which the corresponding $\mathbf{k}$-tuples satisfy that $Q(p_{1})-Q(p_{\mathbf{k}})=\Delta$. The component quasi-distribution, denoted by $\xi_\Delta$, has a quasi-pdf $f_{\xi_\Delta}(\bar{\Delta})=\sum_{\substack{(\psi_{\mathbf{k}}(\mathbf{v}), f_{X,\ldots,X}(\mathbf{v})) \in T_\Delta \\ \bar{\Delta}=\psi_\mathbf{k}(\mathbf{v})}} f_{X,\ldots,X}(\mathbf{v})$, i.e., sum over all $f_{X,\ldots,X}(\mathbf{v})$ such that the pair $(\psi_{\mathbf{k}}(\mathbf{v}), f_{X,\ldots,X}(\mathbf{v}))$ is in the set $T_\Delta$ and the first element of the pair, $\psi_\mathbf{k}(\mathbf{v})$, is equal to $\bar{\Delta}$. The $\mathbf{k}$th, where $\mathbf{k}>2$, central moment kernel distribution, labeled $\Xi_\mathbf{k}$, can be seen as a quasi-mixture distribution comprising an infinite number of component quasi-distributions, $\xi_\Delta$s, each corresponding to a different value of $\Delta$, which ranges from $Q(0)-Q(1)$ to $0$. Each component quasi-distribution has a support of $\left(-\binom{\mathbf{k}}{\frac{3+\left(-1\right)^\mathbf{k}}{2}}^{-1}(-\Delta)^\mathbf{k},\frac{1}{\mathbf{k}}(-\Delta)^\mathbf{k}\right)$.

\end{theorem}

\begin{proof}

The support of $\xi_\Delta$ is the extrema of the function $\psi_\mathbf{k}\left(Q(p_{1}),\Compactcdots,Q(p_{\mathbf{k}})\right)$ subjected to the constraints, $Q(p_{1})<\Compactcdots<Q(p_{\mathbf{k}})$ and $\Delta=Q(p_{1})-Q(p_{\mathbf{k}})$. Using the Lagrange multiplier, the only critical point can be determined at $Q(p_{1})=\Compactcdots=Q(p_{\mathbf{k}})=0$, where $\psi_\mathbf{k}=0$. Other candidates are within the boundaries, i.e., $\psi_\mathbf{k}\left(x_1=Q(p_{1}),x_2=Q(p_{\mathbf{k}}),\Compactcdots,x_\mathbf{k}=Q(p_{\mathbf{k}})\right)$, $\Compactcdots$, $\psi_\mathbf{k}\left(x_1=Q(p_{1}),\Compactcdots,x_i=Q(p_{1}), x_{i+1}=Q(p_{\mathbf{k}}),\Compactcdots,x_\mathbf{k}=Q(p_{\mathbf{k}})\right)$, $\Compactcdots$, $\psi_\mathbf{k}\left(x_1=Q(p_{1}),\Compactcdots,x_{\mathbf{k}-1}=Q(p_{1}),x_\mathbf{k}=Q(p_{\mathbf{k}})\right)$. $\psi_\mathbf{k}\left(x_1=Q(p_{1}),\Compactcdots,x_i=Q(p_{1}),x_{i+1}=Q(p_{\mathbf{k}}),\Compactcdots,x_\mathbf{k}=Q(p_{\mathbf{k}})\right)$ can be divided into $\mathbf{k}$ groups. The $g$th group has the common factor $\left(-1\right)^{g+1}\frac{1}{\mathbf{k}-g+1}$, if $1\leq g\leq \mathbf{k}-1$ and the final $\mathbf{k}$th group is the term $\left(-1\right)^{\mathbf{k}-1}\left(\mathbf{k}-1\right)Q(p_{1})^iQ(p_{\mathbf{\mathbf{k}}})^{\mathbf{k}-i}$. If $\frac{\mathbf{k}+1-i}{2}\le j\le\frac{\mathbf{k}-1}{2}$ and $j+1\le g\le\mathbf{k}-j$, the $g$th group has $i\binom{i-1}{g-j-1}\binom{\mathbf{k}-i}{j}$ terms having the form $\left(-1\right)^{g+1}\frac{1}{\mathbf{k}-g+1}Q(p_{1})^{\mathbf{k}-j}Q(p_{\mathbf{k}})^j$. If $\frac{\mathbf{k}+1-i}{2}\le j\le\frac{\mathbf{k}-1}{2}$ and $\mathbf{k}-j+1\le g\le i+j$, the $g$th group has $i\binom{i-1}{g-j-1}\binom{\mathbf{k}-i}{j}+\left(\mathbf{k}-i\right)\binom{\mathbf{k}-i-1}{j-\mathbf{k}+g-1}\binom{i}{\mathbf{k}-j}$ terms having the form $\left(-1\right)^{g+1}\frac{1}{\mathbf{k}-g+1}Q(p_{1})^{\mathbf{k}-j}Q(p_{\mathbf{k}})^j$. If $0\leq j<\frac{\mathbf{k}+1-i}{2}$ and $j+1\le g\le i+j$, the $g$th group has $i\binom{i-1}{g-j-1}\binom{\mathbf{k}-i}{j}$ terms having the form $\left(-1\right)^{g+1}\frac{1}{\mathbf{k}-g+1}Q(p_{1})^{\mathbf{k}-j}Q(p_{\mathbf{k}})^j$. If $\frac{\mathbf{k}}{2}\leq j\leq\mathbf{k}$ and $\mathbf{k}-j+1\leq g\leq j$, the $g$th group has $\left(\mathbf{k}-i\right)\binom{\mathbf{k}-i-1}{j-\mathbf{k}+g-1}\binom{i}{\mathbf{k}-j}$ terms having the form $\left(-1\right)^{g+1}\frac{1}{\mathbf{k}-g+1}Q(p_{1})^{\mathbf{k}-j}Q(p_{\mathbf{k}})^j$. If $\frac{\mathbf{k}}{2}\leq j\leq\mathbf{k}$ and $j+1\leq g\leq j+i<\mathbf{k}$, the $g$th group has $i\binom{i-1}{g-j-1}\binom{\mathbf{k}-i}{j}+\left(\mathbf{k}-i\right)\binom{\mathbf{k}-i-1}{j-\mathbf{k}+g-1}\binom{i}{\mathbf{k}-j}$ terms having the form $\left(-1\right)^{g+1}\frac{1}{\mathbf{k}-g+1}Q(p_{1})^{\mathbf{k}-j}Q(p_{\mathbf{k}})^j$. So, if $i+j=\mathbf{k}$, $\frac{\mathbf{k}}{2}\leq j\leq\mathbf{k}$, $0\le i\le\frac{\mathbf{k}}{2}$, the summed coefficient of $Q(p_{1})^iQ(p_{\mathbf{k}})^{\mathbf{k}-i}$ is $\left(-1\right)^{\mathbf{k}-1}\left(\mathbf{k}-1\right)+\sum_{g=i+1}^{\mathbf{k}-1}{\left(-1\right)^{g+1}\frac{1}{\mathbf{k}-g+1}\left(\mathbf{k}-i\right)\binom{\mathbf{k}-i-1}{g-i-1}}+\sum_{g=\mathbf{k}-i+1}^{\mathbf{k}-1}{\left(-1\right)^{g+1}\frac{1}{\mathbf{k}-g+1}i\binom{i-1}{g-\mathbf{k}+i-1}}=\left(-1\right)^{\mathbf{k}-1}\left(\mathbf{k}-1\right)+\left(-1\right)^{\mathbf{k}+1}+\left(\mathbf{k}-i\right)\left(-1\right)^\mathbf{k}+\left(-1\right)^\mathbf{k}\left(i-1\right)=\left(-1\right)^{\mathbf{k}+1}$. The summation identities are 
$\sum_{g=i+1}^{\mathbf{k}-1}{\left(-1\right)^{g+1}\frac{1}{\mathbf{k}-g+1}\left(\mathbf{k}-i\right)\binom{\mathbf{k}-i-1}{g-i-1}}=\left(\mathbf{k}-i\right)\int_{0}^{1}{\sum_{g=i+1}^{\mathbf{k}-1}{\left(-1\right)^{g+1}\binom{\mathbf{k}-i-1}{g-i-1}}t^{\mathbf{k}-g}dt}=\left(\mathbf{k}-i\right)\int_{0}^{1}\left(\left(-1\right)^i\left(t-1\right)^{\mathbf{k}-i-1}-\left(-1\right)^{\mathbf{k}+1}\right)dt=\left(\mathbf{k}-i\right)\left(\frac{\left(-1\right)^\mathbf{k}}{i-\mathbf{k}}+\left(-1\right)^\mathbf{k}\right)=\left(-1\right)^{\mathbf{k}+1}+\left(\mathbf{k}-i\right)\left(-1\right)^\mathbf{k}$ and $\sum_{g=\mathbf{k}-i+1}^{\mathbf{k}-1}{\left(-1\right)^{g+1}\frac{1}{\mathbf{k}-g+1}i\binom{i-1}{g-\mathbf{k}+i-1}}=\int_{0}^{1}{\sum_{g=\mathbf{k}-i+1}^{\mathbf{k}-1}{\left(-1\right)^{g+1}i\binom{i-1}{g-\mathbf{k}+i-1}}t^{\mathbf{k}-g}dt}=\int_{0}^{1}\left(i\left(-1\right)^{\mathbf{k}-i}\left(t-1\right)^{i-1}-i\left(-1\right)^{\mathbf{k}+1}\right)dt=\left(-1\right)^\mathbf{k}\left(i-1\right)$. If $0\leq j<\frac{\mathbf{k}+1-i}{2}$ and $i=\mathbf{k}$, $\psi_\mathbf{k}=0$. If $\frac{\mathbf{k}+1-i}{2}\le j\le\frac{\mathbf{k}-1}{2}$ and $\frac{\mathbf{k}+1}{2}\le i\le \mathbf{k}-1$, the summed coefficient of $Q(p_1)^i Q(p_{\mathbf{k}})^{\mathbf{k}-i}$ is $\left(-1\right)^{\mathbf{k}-1}\left(\mathbf{k}-1\right)+\sum_{g=\mathbf{k}-i+1}^{\mathbf{k}-1}{\left(-1\right)^{g+1}\frac{1}{\mathbf{k}-g+1}i\binom{i-1}{g-\mathbf{k}+i-1}}+\\\sum_{g=i+1}^{\mathbf{k}-1}{\left(-1\right)^{g+1}\frac{1}{\mathbf{k}-g+1}\ \left(\mathbf{k}-i\right)\binom{\mathbf{k}-i-1}{g-i-1}}$, the same as above. If $i+j< \mathbf{k}$, since $\binom{i}{\mathbf{k}-j}=0$, the related terms can be ignored, so, using the binomial theorem and beta function, the summed coefficient of $Q(p_1)^{k-j}Q(p_{\mathbf{k}})^{j}$ is $\sum_{g=j+1}^{i+j}{\left(-1\right)^{g+1}\frac{1}{\mathbf{k}-g+1}i\binom{i-1}{g-j-1}\binom{\mathbf{k}-i}{j}}=i\binom{\mathbf{k}-i}{j}\int_{0}^{1}{\sum_{g=j+1}^{i+j}{\left(-1\right)^{g+1}\binom{i-1}{g-j-1}}t^{\mathbf{k}-g}dt}=\binom{\mathbf{k}-i}{j}i\int_{0}^{1}\left(\left(-1\right)^jt^{\mathbf{k}-j-1}\left(\frac{t}{t-1}\right)^{1-i}\right)dt=\binom{\mathbf{k}-i}{j}i\frac{\left(-1\right)^{j+i+1}\Gamma\left(i\right)\Gamma\left(\mathbf{k}-j-i+1\right)}{\Gamma\left(\mathbf{k}-j+1\right)}=\frac{\left(-1\right)^{j+i+1}i!\left(\mathbf{k}-j-i\right)!\left(\mathbf{k}-i\right)!}{\left(\mathbf{k}-j\right)!j!\left(\mathbf{k}-j-i\right)!}=\left(-1\right)^{j+i+1}\frac{i!\left(\mathbf{k}-i\right)!}{\mathbf{k}!}\frac{\mathbf{k}!}{\left(\mathbf{k}-j\right)!j!}=\binom{\mathbf{k}}{i}^{-1}\left(-1\right)^{1+i}\binom{\mathbf{k}}{j}\left(-1\right)^{j}$. 

According to the binomial theorem, the coefficient of $Q(p_1)^i Q(p_\mathbf{k})^{\mathbf{k}-i}$ in $\binom{\mathbf{k}}{i}^{-1}\left(-1\right)^{1+i}\left(Q(p_1)-Q(p_\mathbf{k})\right)^\mathbf{k}$ is $\binom{\mathbf{k}}{i}^{-1}\left(-1\right)^{1+i}\binom{\mathbf{k}}{i}\left(-1\right)^{\mathbf{k}-i}=\left(-1\right)^{\mathbf{k}+1}$, same as the above summed coefficient of $Q(p_{1})^iQ(p_{\mathbf{k}})^{\mathbf{k}-i}$, if $i+j=\mathbf{k}$. If $i+j<k$, the coefficient of $Q(p_1)^{\mathbf{k}-j} Q(p_\mathbf{k})^{j}$ is $\binom{\mathbf{k}}{i}^{-1}\left(-1\right)^{1+i}\binom{\mathbf{k}}{j}\left(-1\right)^{j}$, same as the corresponding summed coefficient of $Q(p_1)^{\mathbf{k}-j} Q(p_\mathbf{k})^{j}$. Therefore, $\psi_\mathbf{k}\left(x_1=Q(p_1),\ldots,x_i=Q(p_1),x_{i+1}=Q(p_\mathbf{k}),\ldots,x_\mathbf{k}=Q(p_\mathbf{k})\right)=\binom{\mathbf{k}}{i}^{-1}\left(-1\right)^{1+i}\left(Q(p_1)-Q(p_\mathbf{k})\right)^\mathbf{k}$, the maximum and minimum of $\psi_\mathbf{k}$ follow directly from the properties of the binomial coefficient.

\end{proof}

The component quasi-distribution, $\xi_\Delta$, is closely related to $\Xi_\Delta$, which is the pairwise difference distribution, since $\sum_{\bar{\Delta}=-\binom{\mathbf{k}}{\frac{3+\left(-1\right)^\mathbf{k}}{2}}^{-1}\left(-\Delta\right)^\mathbf{k}}^{\frac{1}{\mathbf{k}}(-\Delta)^\mathbf{k}}{f_{\xi_\Delta}(\bar{\Delta})}=f_{\Xi_\Delta}(\Delta)$. Recall that Theorem \ref{scmkdti} established that $f_{\Xi_\Delta}(\Delta)$ is monotonic increasing with a mode at zero if the original distribution is unimodal, $f_{\Xi_{-\Delta}}(-\Delta)$ is thus monotonic decreasing with a mode at zero. In general, if assuming the shape of $\xi_\Delta$ is uniform, $\Xi_\mathbf{k}$ is monotonic left and right around zero. The median of $\Xi_\mathbf{k}$ also exhibits a strong tendency to be close to zero, as it can be cast as a weighted mean of the medians of $\xi_\Delta$. When $-\Delta$ is small, all values of $\xi_\Delta$ are close to zero, resulting in the median of $\xi_\Delta$ being close to zero as well. When $-\Delta$ is large, the median of $\xi_\Delta$ depends on its skewness, but the corresponding weight is much smaller, so even if $\xi_\Delta$ is highly skewed, the median of $\Xi_\mathbf{k}$ will only be slightly shifted from zero. Denote the median of $\Xi_\mathbf{k}$ as $m\mathbf{k}m$, for the five parametric distributions here, $|m\mathbf{k}m|$s are all $\leq 0.1\sigma$ for $\Xi_3$ and $\Xi_4$, where $\sigma$ is the standard deviation of $\Xi_\mathbf{k}$ (SI Dataset S1). Assuming $m\mathbf{k}m=0$, for the even ordinal central moment kernel distribution, the average probability density on the left side of zero is greater than that on the right side, since $\frac{\frac{1}{2}}{\binom{\mathbf{k}}{2}^{-1}\left(Q\left(0\right)-Q\left(1\right)\right)^\mathbf{k}}>\frac{\frac{1}{2}}{\frac{1}{\mathbf{k}}\left(Q\left(0\right)-Q\left(1\right)\right)^\mathbf{k}}$. This means that, on average, the inequality $f(Q(\epsilon))\geq f(Q(1-\epsilon))$ holds. For the odd ordinal distribution, the discussion is more challenging since it is generally symmetric. Just consider $\Xi_3$, let $x_1=Q(p_i)$ and $x_3=Q(p_j)$, changing the value of $x_2$ from $Q(p_i)$ to $Q(p_j)$ will monotonically change the value of $\psi_3(x_1,x_2,x_3)$, since $\frac{\partial \psi_3(x_1,x_2,x_3)}{\partial x_2}=-\frac{x_1^2}{2}-x_1 x_2+2 x_1 x_3+x_2^2-x_2 x_3-\frac{x_3^2}{2}$, $-\frac{3}{4}\left(x_1-x_3\right)^2\le\frac{\partial \psi_3(x_1,x_2,x_3)}{\partial x_2}\le-\frac{1}{2}\left(x_1-x_3\right)^2\le0$. If the original distribution is right-skewed, $\xi_\Delta$ will be left-skewed, so, for $\Xi_3$, the average probability density of the right side of zero will be greater than that of the left side, which means, on average, the inequality $f(Q(\epsilon))\le f(Q(1-\epsilon))$ holds. In all, the monotonic decreasing of the negative pairwise difference distribution guides the general shape of the $\mathbf{k}$th central moment kernel distribution, $\mathbf{k}>2$, forcing it to be unimodal-like with the mode and median close to zero, then, the inequality $f(Q(\epsilon))\le f(Q(1-\epsilon))$ or $f(Q(\epsilon))\geq f(Q(1-\epsilon))$ holds in general. If a distribution is $\nu$th $\gamma$-ordered and all of its central moment kernel distributions are also $\nu$th $\gamma$-ordered, it is called completely $\nu$th $\gamma$-ordered. %Although strict complete $\nu$th orderliness is difficult to prove, following the same logic as discussed in REDS I, even if the inequality may be violated in a small range, the mean-$\text{SWA}_{\epsilon}$-median inequality remains valid, in most cases, for the central moment kernel distribution.

%as discussed in Subsection \ref{II}

Another crucial property of the central moment kernel distribution, location invariant, is introduced in the next theorem. %The proof is provided in the SI Text.
\begin{theorem}\label{kkd} $\psi_\mathbf{k}\left(x_1=\lambda x_1+\mu,\Compactcdots,x_\mathbf{k}=\lambda x_\mathbf{k}+\mu\right)=\lambda^\mathbf{k}\psi_\mathbf{k}\left(x_1,\Compactcdots,x_\mathbf{k}\right)$.
\end{theorem}
\begin{proof}

Recall that for the $\mathbf{k}$th central moment, the kernel is $\psi_\mathbf{k}\left(x_1,\ldots,x_\mathbf{k}\right)=\sum_{j=0}^{\mathbf{k}-2}{\left(-1\right)^j\left(\frac{1}{\mathbf{k}-j}\right)\sum\left(x_{i_1}^{\mathbf{k}-j}x_{i_2}\ldots x_{i_{j+1}}\right)}+\left(-1\right)^{\mathbf{k}-1}\left(\mathbf{k}-1\right)x_1\ldots x_\mathbf{k}$, where the second summation is over $i_1,\ldots,i_{j+1}=1$ to $\mathbf{k}$ with $i_1\neq i_2\neq\ldots\neq i_{j+1}$ and $i_2<i_3<\ldots<i_{j+1}$ \cite{heffernan1997unbiased}.

$\psi_\mathbf{k}$ consists of two parts. The first part, $\sum_{j=0}^{\mathbf{k}-2}{\left(-1\right)^j\left(\frac{1}{\mathbf{k}-j}\right)\sum\left(x_{i_1}^{\mathbf{k}-j}x_{i_2}\ldots x_{i_{j+1}}\right)}$, involves a double summation over certain terms. The second part, $\left(-1\right)^{\mathbf{k}-1}\left(\mathbf{k}-1\right)x_1\ldots x_\mathbf{k}$, carries an alternating sign $\left(-1\right)^{\mathbf{k}-1}$ and involves multiplication of the constant $\mathbf{k}-1$ with the product of all the $x$ variables, $x_1x_2\ldots x_\mathbf{k}$. Consider each multiplication cluster $\left(-1\right)^j\left(\frac{1}{\mathbf{k}-j}\right)\sum\left(x_{i_1}^{\mathbf{k}-j}x_{i_2}\ldots x_{i_{j+1}}\right)$ for $j$ ranging from $0$ to $\mathbf{k}-2$ in the first part. Let each cluster form a single group. The first part can be divided into $\mathbf{k}-1$ groups. Combine this with the second part $\left(-1\right)^{\mathbf{k}-1}\left(\mathbf{k}-1\right)x_1\ldots x_\mathbf{k}$. Together, the terms of $\psi_\mathbf{k}$ can be divided into a total of $\mathbf{k}$ groups. From the $1$st to $\mathbf{k}-1$th group, the $g$th group has $\binom{\mathbf{k}}{g}\binom{g}{1}$ terms having the form $\left(-1\right)^{g+1}\frac{1}{\mathbf{k}-g+1}x_{i_1}^{\mathbf{k}-g+1}x_{i_2}\ldots x_{i_g}$. The final $\mathbf{k}$th group is the term $\left(-1\right)^{\mathbf{k}-1}\left(\mathbf{k}-1\right)x_1\Compactcdots x_\mathbf{k}$. 
There are two ways to divide $\psi_\mathbf{k}$ into $\mathbf{k}$ groups according to the form of each term. The first choice is, if $\mathbf{k}\neq g$, the $g$th group of $\psi_\mathbf{k}$ has $\binom{\mathbf{k}-l}{g-l}$ terms having the form $\left(-1\right)^{g+1}\frac{1}{\mathbf{k}-g+1}x_{i_1}^{\mathbf{k}-g+1}x_{i_2}\Compactcdots x_{i_l}x_{i_{l+1}}\ldots x_{i_g}$, where $x_{i_1}, x_{i_2},\Compactcdots,x_{i_l}$ are fixed, $x_{i_{l+1}},\Compactcdots,x_{i_{g}}$ are selected such that $i_{l+1},\Compactcdots,\ i_{g}\neq{i_1},{i_2},\Compactcdots,{i_l}$ and $i_{l+1}\neq\ldots\neq i_{g}$. Define another function $\Psi_\mathbf{k}\left(x_{i_1},x_{i_2},\Compactcdots,x_{i_l},x_{{i_{l+1}}},\Compactcdots,x_{i_{g}}\right)=\left(\lambda x_{i_1}+\mu\right)^{\mathbf{k}-g+1}\left(\lambda x_{i_2}+\mu\right)\Compactcdots\left(\lambda x_{i_l}+\mu\right)\left(\lambda x_{i_{l+1}}+\mu\right)\Compactcdots\left(\lambda x_{i_{g}}+\mu\right)$, the first group of $\Psi_\mathbf{k}$ is $\lambda^{\mathbf{k}}x_{i_1}\Compactcdots x_{i_l} x_{i_{l+1}}\Compactcdots x_{i_{g}}$, the $h$th group of $\Psi_\mathbf{k}$, $h>1$, has $\binom{\mathbf{k}-g+1}{\mathbf{k}-h-l+2}$ terms having the form $\lambda^{\mathbf{k}-h+1}\mu^{h-1}x_{i_1}^{\mathbf{k}-h-l+2}x_{i_2}\Compactcdots x_{i_l}$. Transforming $\psi_\mathbf{k}$ by $\Psi_\mathbf{k}$, then combing all terms with $\lambda^{\mathbf{k}-h+1}\mu^{h-1}x_{i_1}^{\mathbf{k}-h-l+2}x_{i_2}\Compactcdots x_{i_l}$, $\mathbf{k}-h-l+2>1$, the summed coefficient is ${S1}_l=\sum_{g=l}^{h+l-1}{\left(-1\right)^{g+1}\frac{1}{\mathbf{k}-g+1}\binom{\mathbf{k}-g+1}{\mathbf{k}-h-l+2}\binom{\mathbf{k}-l}{g-l}}=\sum_{g=l}^{h+l-1}{\left(-1\right)^{g+1}\frac{\left(\mathbf{k}-l\right)!}{\left(h+l-g-1\right)!\left(\mathbf{k}-h-l+2\right)!\left(g-l\right)!}}=0,$ since the summation is starting from $l$, ending at $h+l-1$, the first term includes the factor $g-l=0$, the final term includes the factor $h+l-g-1=0$, the terms in the middle are also zero due to the factorial property. 

%$x_1^{\mathbf{k}-h-l+2}\neq x_1$

Another possible choice is the $g$th group of $\psi_\mathbf{k}$ has $\left(\mathbf{k}-h\right)\binom{h-1}{g-\mathbf{k}+h-1}$ terms having the form 

$\left(-1\right)^{g+1}\frac{1}{\mathbf{k}-g+1}x_{i_1}x_{i_2}\Compactcdots x_{i_j}^{\mathbf{k}-g+1}\Compactcdots x_{i_{\mathbf{k}-h+1}}x_{i_{\mathbf{k}-h+2}}\Compactcdots x_{i_{g}}$, provided that $\mathbf{k}\neq g$, $2\le j\le \mathbf{k}-h+1$, where $x_{i_1},\ldots,x_{{i_{\mathbf{k}-h+1}}}$ are fixed, $x_{i_j}^{\mathbf{k}-g+1}$ and $x_{i_{\mathbf{k}-h+2}},\Compactcdots,x_{i_{g}}$ are selected such that ${i_{\mathbf{k}-h+2}},\Compactcdots,\ {i_{g}}\neq{i_1},{i_2},\Compactcdots,{i_{\mathbf{k}-h+1}}$ and ${i_{\mathbf{k}-h+2}}\neq\ldots\neq {i_{g}}$. Transforming these terms by $\Psi_\mathbf{k}\left(x_{i_1},x_{i_2},\ldots,x_{i_j},\ldots,x_{i_{\mathbf{k}-h+1}},x_{i_{\mathbf{k}-h+2}},\ldots,x_{i_{g}}\right)=\\\left(\lambda x_{i_1}+\mu\right)\left(\lambda x_{i_2}+\mu\right)\Compactcdots\left(\lambda x_{i_j}+\mu\right)^{\mathbf{k}-g+1}\Compactcdots\left(\lambda x_{i_{\mathbf{k}-h+1}}+\mu\right)\left(\lambda x_{i_{\mathbf{k}-h+2}}+\mu\right)\Compactcdots\left(\lambda x_{i_{g}}+\mu\right)$, then there are $\mathbf{k}-g+1$ terms having the form $\lambda^{\mathbf{k}-h+1}\mu^{h-1}x_{i_1}x_{i_2}\ldots x_{i_{\mathbf{k}-h+1}}$. Transforming the final $\mathbf{k}$th group of $\psi_\mathbf{k}$ by $\Psi_\mathbf{k}\left(x_1,\ldots,x_{\mathbf{k}}\right)=\left(\lambda x_1+\mu\right)\Compactcdots\left(\lambda x_{\mathbf{k}}+\mu\right)$, then, there is one term having the form $\left(-1\right)^{\mathbf{k}-1}\left(\mathbf{k}-1\right)\lambda^{\mathbf{k}-h+1}\mu^{h-1}x_1x_2\ldots x_{\mathbf{k}-h+1}$. Another possible combination is that the $g$th group of $\psi_\mathbf{k}$ contains $\left(g-\mathbf{k}+h-1\right)\binom{h-1}{g-\mathbf{k}+h-1}$ terms having the form $\left(-1\right)^{g+1}\frac{1}{\mathbf{k}-g+1}x_{i_1}x_{i_2}\Compactcdots x_{i_{\mathbf{k}-h+1}}x_{i_{\mathbf{k}-h+2}}\Compactcdots x_{i_j}^{\mathbf{k}-g+1}\Compactcdots x_{i_{g}}$. Transforming these terms by $\Psi_\mathbf{k}\left(x_{i_1},x_{i_2},\ldots,x_{i_{\mathbf{k}-h+1}},x_{i_{\mathbf{k}-h+2}},\ldots,x_{i_j},\ldots,x_{i_{g}}\right)=\\\left(\lambda x_{i_1}+\mu\right)\left(\lambda x_{i_2}+\mu\right)\Compactcdots\left(\lambda x_{i_{\mathbf{k}-h+1}}+\mu\right)\left(\lambda x_{i_{\mathbf{k}-h+2}}+\mu\right)\Compactcdots\left(\lambda x_{i_j}+\mu\right)^{\mathbf{k}-g+1}\Compactcdots\left(\lambda x_{i_{g}}+\mu\right)$, then there is only one term having the form $\lambda^{\mathbf{k}-h+1}\mu^{h-1}x_{i_1}x_{i_2}\ldots x_{i_{\mathbf{k}-h+1}}$. The above summation $S1_l$ should also be included, i.e., $x_{i_1}^{\mathbf{k}-h-l+2}=x_{i_1}$, $\mathbf{k}=h+l-1$. So, combing all terms with $\lambda^{\mathbf{k}-h+1}\mu^{h-1}x_{i_1}x_{i_2}\ldots x_{i_{\mathbf{k}-h+1}}$, according to the binomial theorem, the summed coefficient is $S2_l=\sum_{g=\mathbf{k}-h+1}^{\mathbf{k}-1}{\left(-1\right)^{g+1}\binom{h-1}{g-\mathbf{k}+h-1}\left(\mathbf{k}-h+1+\frac{g-\mathbf{k}+h-1}{\mathbf{k}-g+1}\right)}+\left(-1\right)^{\mathbf{k}-1}\left(\mathbf{k}-1\right)=\left(\mathbf{k}-h+1\right)\sum_{g=\mathbf{k}-h+1}^{\mathbf{k}-1}{\left(-1\right)^{g+1}\binom{h-1}{g-\mathbf{k}+h-1}}+\sum_{g=\mathbf{k}-h+1}^{\mathbf{k}-1}{\left(-1\right)^{g+1}\binom{h-1}{g-\mathbf{k}+h-1}\left(\frac{g-\mathbf{k}+h-1}{\mathbf{k}-g+1}\right)}+\left(-1\right)^{\mathbf{k}-1}\left(\mathbf{k}-1\right)=(-1)^\mathbf{k} (\mathbf{k}-h+1)+(h-2) (-1)^\mathbf{k}+\left(-1\right)^{\mathbf{k}-1}\left(\mathbf{k}-1\right)=0$. The summation identities required are $\sum_{g=\mathbf{k}-h+1}^{\mathbf{k}-1}{\left(-1\right)^{g+1}\binom{h-1}{g-\mathbf{k}+h-1}}=(-1)^\mathbf{k}$ and $\sum_{g=\mathbf{k}-h+1}^{\mathbf{k}-1}{\left(-1\right)^{g+1}\binom{h-1}{g-\mathbf{k}+h-1}\left(\frac{g-\mathbf{k}+h-1}{\mathbf{k}-g+1}\right)}=(h-2)(-1)^\mathbf{k}$. These two summation identities are proven in Lemma 4 and 5 in the SI Text.%\ref{11} and \ref{22}. 

Thus, no matter in which way, all terms including $\mu$ can be canceled out. The proof is complete by noticing that the remaining part is $\lambda^\mathbf{k}\psi_\mathbf{k}\left(x_1,\Compactcdots,x_\mathbf{k}\right)$.

\end{proof}

A direct result of Theorem \ref{kkd} is that, $\text{WHL}\mathbf{k}m$ after standardization is invariant to location and scale. So, the weighted H-L standardized $\mathbf{k}$th moment is defined to be
\begin{align*}
\text{WHL}s\mathbf{k}m_{\epsilon=\min{(\epsilon_1,\epsilon_2)},k_1,k_2,\gamma_1,\gamma_2 ,n}\coloneqq \frac{\text{WHL}\mathbf{k}m_{k_1,\epsilon_1,\gamma_1,n}}{(\text{WHL}var_{k_2,\epsilon_2,\gamma_2,n})^{\mathbf{k}/2}}
\text{.}
\end{align*}% where $\epsilon_{U_k}=1-\left(1-\epsilon\right)^\frac{1}{k}$.

To avoid confusion, it should be noted that the robust location estimations of the kernel distributions discussed in this paper differ from the approach taken by Joly and Lugosi (2016) \cite{joly2016robust}, which is computing the median of all $U$-statistics from different disjoint blocks. Compared to bootstrap median $U$-statistics, this approach can produce two additional kinds of finite sample bias, one arises from the limited numbers of blocks, another is due to the size of the $U$-statistics (consider the mean of all $U$-statistics from different disjoint blocks, it is definitely not identical to the original $U$-statistic, except when the kernel is the Hodges-Lehmann kernel). Laforgue, Clemencon, and Bertail (2019)'s median of randomized $U$-statistics \cite{laforgue2019medians} is more sophisticated and can overcome the limitation of the number of blocks, but the second kind of bias remains unsolved. 
%The generalized Catoni $M$-estimator (GCM) has excellent performance, both bias and variance, for heavy-tailed distributions (Table \ref{tab:comparison}, SI Dataset S1) \cite{catoni2012challenging,XuCatoni}, but it often falls returning a solution for platykurtic distributions. In addition, its influence function is not necessarily bounded \cite{mathieu2022concentration}, so its robustness is much smaller than those of other robust mean estimators discussed in REDS 1. 

\section*{Congruent Distribution}\label{cd} In the realm of nonparametric statistics, the relative differences, or orders, of robust estimators are of primary importance. A key implication of this principle is that when there is a shift in the parameters of the underlying distribution, all nonparametric estimates should asymptotically change in the same direction, if they are estimating the same attribute of the distribution. If, on the other hand, the mean suggests an increase in the location of the distribution while the median indicates a decrease, a contradiction arises. It is worth noting that such contradiction is not possible for any $LL$-statistics in a location-scale distribution, as explained in Theorem 2 and 18 in REDS I. However, it is possible to construct counterexamples to the aforementioned implication in a shape-scale distribution. In the case of the Weibull distribution, its quantile function is $Q_{Wei}\left(p\right)=\lambda  (-\ln (1-p))^{1/\alpha}$, where $0\leq p\leq1$, $\alpha>0$, $\lambda>0$, $\lambda$ is a scale parameter, $\alpha$ is a shape parameter, $\ln$ is the natural logarithm function. Then, $m=\lambda \sqrt[\alpha ]{\ln (2)}$, $\mu=\lambda \Gamma \left(1+\frac{1}{\alpha }\right)$, where $\Gamma$ is the gamma function. When $\alpha=1$, $m=\lambda \ln (2)\approx0.693\lambda$, $\mu=\lambda$, when $\alpha=\frac{1}{2}$, $m=\lambda \ln ^2(2)\approx0.480\lambda$, $\mu=2\lambda$, the mean increases as $\alpha$ changes from $1$ to $\frac{1}{2}$, but the median decreases. In the last section, the fundamental role of quantile average was demonstrated by using the method of classifying distributions through the signs of derivatives. To avoid such scenarios, this method can also be used. Let the quantile average function of a parametric distribution be denoted as $\text{QA}\left(\epsilon,\gamma,\alpha_1,\Compactcdots,\alpha_i,\Compactcdots,\alpha_k\right)$, where $\alpha_i$ represent the parameters of the distribution, then, a distribution is $\gamma$-congruent if and only if the sign of $\frac{\partial{\text{QA}}}{\partial{\alpha_i}}$ remains the same for all $0\leq\epsilon\leq\frac{1}{1+\gamma}$. If $\frac{\partial{\text{QA}}}{\partial{\alpha_i}}$ is equal to zero or undefined, it can be considered both positive and negative, and thus does not impact the analysis. A distribution is completely $\gamma$-congruent if and only if it is $\gamma$-congruent and all its central moment kernel distributions are also $\gamma$-congruent. Setting $\gamma=1$ constitutes the definitions of congruence and complete congruence. Replacing the QA with QHLM gives the definition of $\gamma$-$U$-congruence. Chebyshev's inequality implies that, for any probability distributions with finite second moments, as the parameters change, even if some $LL$-statistics change in a direction different from that of the population mean, the magnitude of the changes in the $LL$-statistics remains bounded compared to the changes in the population mean. Furthermore, distributions with infinite moments can be $\gamma$-congruent, since the definition is based on the quantile average, not the population mean.

The following theorems show the conditions that a distribution is congruent or $\gamma$-congruent.
\begin{theorem}\label{sscon}A symmetric distribution is always congruent and $U$-congruent.\end{theorem}\begin{proof}As shown in Theorem 2 and Theorem 18 in REDS I, for any symmetric distribution, all symmetric quantile averages and all SQHLMs conincide. The conclusion follows immediately.\end{proof}

\begin{theorem}\label{lcscon}A positive definite location-scale distribution is always $\gamma$-congruent.\end{theorem}\begin{proof}As shown in Theorem 2, for a location-scale distribution, any quantile average can be expressed as  $\lambda \mathrm{QA}_{0}(\epsilon,\gamma)+\mu$. Therefore, the derivatives with respect to the parameters $\lambda$ or $\mu$ are always positive. By application of the definition, the desired outcome is obtained.\end{proof}

For the Pareto distribution, $\frac{\partial{Q}}{\partial{\alpha}}=\frac{x_m(1-p)^{-1/\alpha} \ln (1-p)}{\alpha^2}$. Since $\ln (1-p)<0$ for all $0<p<1$, $(1-p)^{-1/\alpha}>0$ for all $0<p<1$ and $\alpha>0$, so $\frac{\partial{Q}}{\partial{\alpha}}<0$, and therefore $\frac{\partial{\text{QA}}}{\partial{\alpha}}<0$, the Pareto distribution is $\gamma$-congruent. It is also $\gamma$-$U$-congruent, since $\gamma m$HLM can also express as a function of ${Q}(p)$. For the lognormal distribution, $\frac{\partial{\text{QA}}}{\partial{\sigma}}=\frac{1}{2} \biggl(\sqrt{2} \text{erfc}^{-1}(2\gamma \epsilon) \biggl(-e^{\frac{\sqrt{2} \mu-2 \sigma \text{erfc}^{-1}(2 \gamma\epsilon )}{\sqrt{2}}}\biggr)+\biggl(-\sqrt{2}\biggr) \text{erfc}^{-1}(2 (1-\epsilon)) e^{\frac{\sqrt{2} \mu-2 \sigma \text{erfc}^{-1}(2 (1-\epsilon))}{\sqrt{2}}}\biggr) $. Since the inverse complementary error function is positive when the input is smaller than 1, and negative when the input is larger than 1, and symmetry around 1, if $0\leq\gamma\leq1$, $\text{erfc}^{-1}(2 \gamma\epsilon)\geq-\text{erfc}^{-1}(2-2 \epsilon)$, $e^{\mu-\sqrt{2} \sigma \text{erfc}^{-1}(2-2 \epsilon)}>e^{\mu-\sqrt{2} \sigma \text{erfc}^{-1}(2 \gamma\epsilon)}$. Therefore, if $0\leq\gamma\leq1$, $\frac{\partial{\text{QA}}}{\partial{\sigma}}>0$, the lognormal distribution is $\gamma$-congruent. Theorem \ref{sscon} implies that the generalized Gaussian distribution is congruent and $U$-congruent. For the Weibull distribution, when $\alpha$ changes from 1 to $\frac{1}{2}$, the average probability density on the left side of the median increases, since $\frac{\frac{1}{2}}{\lambda \ln (2)}<\frac{\frac{1}{2}}{\lambda \ln ^2(2)}$, but the mean increases, indicating that the distribution is more heavy-tailed, the probability density of large values will also increase. So, the reason for non-congruence of the Weibull distribution lies in the simultaneous increase of probability densities on two opposite sides as the shape parameter changes: one approaching the bound zero and the other approaching infinity. Note that the gamma distribution does not have this issue, Numerical results indicate that it is likely to be congruent.

The next theorem shows an interesting relation between congruence and the central moment kernel distribution.
\begin{theorem}\label{hhhh1}The second central moment kernal distribution derived from a continuous location-scale unimodal distribution is always $\gamma$-congruent.\end{theorem}\begin{proof}Theorem \ref{kkd} shows that the central moment kernel distribution generated from a location-scale distribution is also a location-scale distribution. Theorem \ref{scmkdti} shows that it is positively definite. Implementing Theorem 12 in REDS 1 yields the desired result.\end{proof}

Although some parametric distributions are not congruent, as shown in REDS 1. In REDS 1, Theorem 12 establishes that $\gamma$-congruence always holds for a positive definite location-scale family distribution and thus for the second central moment kernel distribution generated from a location-scale unimodal distribution as shown in Theorem \ref{hhhh1}. Theorem \ref{tfkumo} demonstrates that all central moment kernel distributions are unimodal-like with mode and median close to zero, as long as they are generated from unimodal distributions. Assuming finite moments and constant $Q(0)-Q(1)$, increasing the mean of a distribution will result in a generally more heavy-tailed distribution, i.e., the probability density of the values close to $Q(1)$ increases, since the total probability density is 1. In the case of the $\mathbf{k}$th central moment kernel distribution, $\mathbf{k}>2$, while the total probability density on either side of zero remains generally constant as the median is generally close to zero and much less impacted by increasing the mean, the probability density of the values close to zero decreases as the mean increases. This transformation will increase nearly all symmetric weighted averages, in the general sense. Therefore, except for the median, which is assumed to be zero, nearly all symmetric weighted averages for all central moment kernel distributions derived from unimodal distributions should change in the same direction when the parameters change. %They are valid measures for nonparametric descriptive statistics.

\vspace*{-10pt}
\section*{Discussion}
Moments, including raw moments, central moments, and standardized moments, are the most common parameters that describe probability distributions. Central moments are preferred over raw moments because they are invariant to translation. In 1947, Hsu and Robbins proved that the arithmetic mean converges completely to the population mean provided the second moment is finite \cite{hsu1947complete}. The strong law of large numbers (proven by Kolmogorov in 1933) \cite{kolmogorov1933sulla} implies that the $\mathbf{k}$th sample central moment is asymptotically unbiased. Recently, fascinating statistical phenomena regarding Taylor's law for distributions with infinite moments have been discovered by Drton and Xiao (2016) \cite{drton2016wald}, Pillai and Meng (2016) \cite{pillai2016unexpected}, Cohen, Davis, and Samorodnitsky (2020) \cite{cohen2020heavy}, and Brown, Cohen, Tang, and Yam (2021) \cite{brown2021taylor}. Lindquist and Rachev (2021) raised a critical question in their inspiring comment to Brown et al's paper \cite{brown2021taylor}: "What are the proper measures for the location, spread, asymmetry, and dependence (association) for random samples with infinite mean?" \cite{lindquist2021taylor}. From a different perspective, this question closely aligns with the essence of Bickel and Lehmann's open question in 1979 \cite{bickel2012descriptive}. They suggested using median, interquartile range, and medcouple \cite{brys2004robust} as the robust versions of the first three moments. While answering this question is not the focus of this paper, it is almost certain that the estimators proposed in this paper will have a place. Since the estimation of central moments can be transformed into the location estimation of a pseudosample, according to the general principle of central limit theorem, the optimal estimator should always has a combinatorial pseudosample size, which explains, in another aspect, why the theory of $U$-statistics allows a minimum-variance unbiased estimator to be derived from each unbiased estimator of an estimable parameter. Similar to the robust version of L-moment \cite{hosking1990moments} being trimmed L-moment \cite{elamir2003trimmed}, central moments now also have their robust nonparametric version, weighted Hodges-Lehmann central moments, based on the complete $U$-congruence of the underlying distribution.

\showmatmethods{} % Display the Materials and Methods section

\vspace*{-3pt}
\subsection*{Software Availability} %Data for Table \ref{tab:comparison} are given in SI Dataset S1-S4. 
The codes used to compute the weighted H-L $\mathbf{k}$th central moment have been deposited in \href {https://github.com/johon-lituobang/REDS_Central_Moments} {github.com/johon-lituobang/REDS}.
\vspace*{-3pt}
\acknow{I sincerely acknowledge the insightful comments from Peter Bickel, which considerably elevating the lucidity and merit of this paper. I am also grateful to Ruodu Wang for pointing out important mistakes regarding the $\gamma$-symmetric distribution.}

\showacknow{} % Display the acknowledgments section

% Bibliography
\bibliography{pnas-sample}

\end{document}